\documentclass[12pt]{article}%\documentclass[12pt,leqno]{article}

\usepackage{amsfonts,mathrsfs}
\usepackage{amssymb,amsmath,amsbsy,amsthm}
\usepackage{dsfont}
\usepackage{blkarray}
\usepackage{tikz}
\usepackage{tkz-euclide}
\usepackage{tikz-cd}
\usetikzlibrary{patterns}
\usepackage{ae}
\usepackage{bm}
\usepackage{graphicx}
\usepackage{float}
\usepackage{cite}
 \usepackage{indentfirst}
\usepackage{lineno}
\usepackage{titlesec}
\usepackage{enumitem}
\usepackage{color}
\usepackage{aliascnt}
\usepackage{varioref}
\usepackage{hyperref}
\hypersetup{colorlinks=true,linkcolor=cyan,anchorcolor=cyan,citecolor=red,CJKbookmarks=True,
,urlcolor=cyan}
\usepackage{cleveref}

\numberwithin{equation}{section}

\def\int{\mbox{\rm int}}

\def\And{\mbox{\rm ~and~}}

\def\({\mbox{\rm (}}\def\){\mbox{\rm )}}

\makeatletter

\newcommand{\Rmnum}[1]{\expandafter\@slowromancap\romannumeral #1@}
\makeatother
\newtheorem{theorem}{Theorem}[section]
\newaliascnt{lemma}{theorem}
\newtheorem{lemma}[lemma]{Lemma}
\aliascntresetthe{lemma}

\newaliascnt{proposition}{theorem}
\newtheorem{proposition}[proposition]{Proposition}
\aliascntresetthe{proposition}

\newaliascnt{fact}{theorem}

\aliascntresetthe{fact}

\newaliascnt{definition}{theorem}

\aliascntresetthe{definition}

\newaliascnt{conjecture}{theorem}

\aliascntresetthe{conjecture}

\newaliascnt{corollary}{theorem}
\newtheorem{corollary}[corollary]{Corollary}
\aliascntresetthe{corollary}

\newaliascnt{claim}{theorem}

\aliascntresetthe{claim}

\newaliascnt{problem}{theorem}

\aliascntresetthe{problem}

\newaliascnt{question}{theorem}

\aliascntresetthe{question}

\newaliascnt{remark}{theorem}

\aliascntresetthe{remark}

\newaliascnt{example}{theorem}

\aliascntresetthe{example}

\newaliascnt{notation}{theorem}

\aliascntresetthe{notation}

\begin{document}
\begin{center}
{\Large\bf Multivariate Tutte polynomials of semimatroids}\\[7pt]
\end{center}
\begin{center}
Houshan Fu\\[5pt]
School of Mathematics and Information Science\\
 Guangzhou University\\
Guangzhou 510006, Guangdong, P. R. China\\[5pt]
Email: fuhoushan@gzhu.edu.cn\\[15pt]
\end{center}
\begin{abstract}
We introduce and investigate multivariate Tutte polynomials, dichromatic polynomials, subset-corank polynomials, size-corank polynomials, and rank generating polynomials of semimatroids, which generalize the corresponding polynomial invariants of graphs and matroids. We primarily establish their deletion-contraction recurrences, basis activities expansions, and various convolution identities. These findings naturally extend Kook-Reiner-Stanton's convolution formula and Kung's convolution-multiplication identities for the Tutte polynomials of graphs and matroids to semimatroids.

\vspace{1ex}
\end{abstract}
\section{Introduction}\label{Sec1}
In this paper, we primarily introduce and investigate multivariate Tutte polynomials, dichromatic polynomials, subset-corank polynomials, size-corank polynomials, and rank generating polynomials of semimatroids, particularly focusing on their alternative descriptions, various convolution formulae, and connections to the standard Tutte polynomial and characteristic polynomial.

Characteristic and Tutte polynomials are the most important and extensively studied polynomial invariants in graphs, matroids, hyperplane arrangements, arithmetic matroids, and semimatroids. The Tutte polynomial originates from the chromatic polynomial, a graph polynomial arising in graph coloring \cite{Tutte1954}. To address the four colour problem, Birkhoff  \cite{Birkhoff1912} introduced the chromatic polynomial for planar graphs in 1912, later generalized to arbitrary graphs by Whitney \cite{Whitney1932,Whitney1932-1}. In 1964, Rota  \cite{Rota1964} defined the characteristic polynomial of matroids as an extension of the chromatic polynomial. Subsequently, Crapo \cite{Crapo1969} further introduced the Tutte polynomial for matroids in terms of internal and external activities of bases, generalizing the characteristic polynomial and showing that its coefficients are nonnegative integers. This combinatorial definition is now known as {\em Crapo's formula} (also referred to as the {\em basis activities expansion}).
\begin{theorem}[\cite{Crapo1969}, Crapo's formula] 
Let $M=(E,r)$ be a matroid. Then the Tutte polynomial $T(M;\lambda,x)$ of $M$ can be expressed as
\[
T(M;\lambda,x)=\sum_{B\text{ is a basis of } M}\lambda^{|{\rm IA}(B)|}x^{|{\rm EA}(B)|},
\]
where ${\rm IA}(B)$ is the set of internally activity elements and ${\rm EA}(B)$ is the set of externally activity elements with respect to a basis $B$ of $M$.
\end{theorem}

In 2007, Ardila \cite{Ardila2007} extended the characteristic and Tutte polynomials to semimatroids. A semimatroid is an abstract structure that captures the dependence properties of affine hyperplane arrangements, generalizing the concept of matroids. This was created by Wachs and Walker \cite{Wachs-Walker1986} in terms of the ``geometric lattice" of closed sets, and later developed a comprehensive theory in terms of subsets of a finite set by Ardila \cite{Ardila2007} (also discovered independently by Kawahara \cite{Kawahara2004}). Ardila focused on Tutte polynomials for semimatroids, establishing a deletion-contraction formula \cite[Proposition 8.2]{Ardila2007} and extending Crapo's formula \cite[Theorem 9.5]{Ardila2007}. However, the characteristic polynomials of semimatroids received limited attention until Fu \cite{Fu2025}, who studied their fundamental properties and further provided convolution formulae for both the Tutte polynomial and the multiplicative characteristic polynomial of semimatroids. Our original motivation is to find a general polynomial that encompasses both the Tutte and characteristic polynomials of semimatroids, and to better understand their equivalent characterizations and convolution identities.

Recently, Sokal \cite{Sokal2005} introduced a multivariate generalization of the Tutte polynomial for graphs and matroids--known to physicists as the Potts-model partition function. The multivariate Tutte polynomial assigns a variable $x_e$ for every element $e$ of a matroid, along with an extra variable $\lambda$. In particular, when all the variables $x_e$  are set to the same variable $x$,  it reduces to a bivariate polynomial, known as the dichromatic polynomial, which is essentially equivalent to the standard Tutte polynomial. In 2014, Br\"ande\'n and Moci \cite{BM2014} generalized this multivariate Tutte polynomial to arithmetic matroids, providing a deletion-contraction recurrence and a generalization of Crapo's formula. 

Inspired by Sokal's work, we introduce a {\em multivariate Tutte polynomial} for a semimatroid $(E,\mathcal{C},r_\mathcal{C})$, defined as follows:
\[
Z(\mathcal{C};\lambda,\bm x):=\sum_{A\in\mathcal{C}}\lambda^{-r_\mathcal{C}(A)}\Big[\prod_{e\in A}x_e\Big]=\sum_{A\in\mathcal{C}}\lambda^{-r_\mathcal{C}(A)}\bm x^A.
\]
which naturally extends the multivariate Tutte polynomial for matroids. For this polynomial, we first present a deletion-contraction formula in \autoref{Mul-DCF} and a generalization of  Crapo's formula in \autoref{Mul-EIA}. In addition, when all the variables $x_e$  are set to the same variable $x$, this specializes to the following {\em dichromatic polynomial}:
\[
Z(\mathcal{C};\lambda,x):=\sum_{A\in\mathcal{C}}\lambda^{-r_\mathcal{C}(A)}x^{|A|}.
\]
As a direct consequence of  \autoref{Mul-DCF} and  \autoref{Mul-EIA}, we derive its deletion-contraction recurrence and basis activities expansion in \autoref{Mul-Dichromatic-IEA}.

Furthermore, we also provide convolution formulae for the multivariate Tutte polynomial of semimatroids in \autoref{Mul-Convolution-Formula}, and for the dichromatic polynomial in \autoref{Mul-Convolution-Formula1}. This extends independent work by Kook, Reiner and Stanton \cite{KRS1999}, and Etienne and Las Vergnas \cite{EL1998}, who established a so-called {\em convolution formula} for the standard Tutte polynomial of a matroid $M=(E,r)$:
\begin{equation}\label{Convolution-Formula}
T(M;\lambda,x)=\sum_{T\subseteq E}T(M|T;0,x)T(M/T;\lambda,0).
\end{equation}
In 2004, Kung discovered a convolution formula for multiplicative characteristic polynomials of matroids. Subsequently, Kung \cite{Kung2010} generalized this result to the subset-corank and size-corank polynomials of matroids, both closely related to the multivariate and standard Tutte polynomials, respectively. In 2015, Wang \cite{Wang2015} introduced M\"obius conjugation of posets as a unified framework for reproving previous convolution formulae and established the first convolution formula for multiplicative characteristic polynomials of hyperplane arrangements. Most recently, a number of convolution formulae for arithmetic Tutte polynomials were developed by Backman and Lenz \cite{BL2017}, and for multivariate arithmetic Tutte polynomials were established by Ma, Jin and Yang \cite{MJY2024}.

Motivated by Kung's work \cite{Kung2010}, we further introduce and study the subset-corank, size-corank, and rank generating polynomials for semimatroids in \autoref{Sec-3}. Note that the subset-corank polynomial is essentially equivalent to the multivariate Tutte polynomial under certain parametrizations, while both the size-corank and rank-generating polynomials are equivalent to the dichromatic polynomial. By applying those results for the multivariate Tutte polynomial of semimatroids to the subset-corank polynomial, we can deduce the corresponding deletion-contraction formula, basis activities expansion, and convolution formulae in \autoref{Pro-Subset-corank-Polynomial}. Similarly, when all the variables $x_e$ are specialized to a common variable $x$, the subset-corank polynomial reduces to the size-corank polynomial. Consequently, the established properties--deletion-contraction formula, basis activities expansion, and convolution formulae--are straightforwardly extended to both the size-corank and rank-generating polynomials, as shown in \autoref{Pro-Size-corank-Polynomial} and \autoref{Prop-Rank-Polynomial}. These findings extend Kung's results in \cite{Kung2010} to semimatroids.

This paper is structured as follows. In \autoref{Sec-2}, we begin by reviewing necessary definitions and results related to semimatroids, followed by the introduction and analysis of their multivariate Tutte polynomials and dichromatic polynomials--with a particular emphasis on their diverse descriptive forms and convolution identities. Building on these findings, we further investigate analogous properties of the subset-corank, size-corank, and rank generating polynomials of semimatroids in \autoref{Sec-3}.
\section{Multivariate Tutte polynomial}\label{Sec-2}
In this section, we first present the multivariate Tutte polynomials and dichromatic polynomials of semimatroids, which extend the corresponding polynomials defined for matroids or graphs. We then investigate their properties pertaining to different descriptive forms and convolution identities.
\subsection{Basic definitions}\label{Sec-2-0}
We begin with the necessary definitions and results connecting semimatroids, which may be needed later. For further reading on this topic, we refer the reader to the literature \cite{Ardila2007}. A {\em semimatroid} is a triple $(E,\mathcal{C},r_{\mathcal{C}})$ consisting of a finite set $E$, a nonempty simplicial complex $\mathcal{C}$ on $E$, and a function $r_\mathcal{C}:\mathcal{C}\to\mathbb{N}$, satisfying the following five properties:
\begin{itemize}[leftmargin=1.4cm]
\item[{\rm (SR1)}] If $X\in\mathcal{C}$, then $0\le r_{\mathcal{C}}(X)\le|X|$.
\item[{\rm (SR2)}] If $X,Y\in\mathcal{C}$ and $X\subseteq Y$, then $r_{\mathcal{C}}(X)\le r_{\mathcal{C}}(Y)$.
\item[{\rm (SR3)}] If $X,Y\in\mathcal{C}$ and $X\cup Y\in\mathcal{C}$, then 
\[
r_{\mathcal{C}}(X\cap Y)+r_{\mathcal{C}}(X\cup Y)\le r_{\mathcal{C}}(X)+r_{\mathcal{C}}(Y).
\]
\item[{\rm (SR4)}] If $X,Y\in\mathcal{C}$ and $r_{\mathcal{C}}(X)=r_{\mathcal{C}}(X\cap Y)$, then $X\cup Y\in\mathcal{C}$.
\item[{\rm (SR5)}] If $X,Y\in\mathcal{C}$ and $r_{\mathcal{C}}(X)<r_{\mathcal{C}}(Y)$, then $X\cup e\in\mathcal{C}$ for some $e\in Y-X$.
\end{itemize}
Every member in $\mathcal{C}$ is referred to as a {\em central set} of the semimatroid $(E,\mathcal{C},r_{\mathcal{C}})$. $E$, $\mathcal{C}$ and $r_{\mathcal{C}}$ are called the {\em ground set}, the {\em collection of central sets} and the {\em rank function} of $(E,\mathcal{C},r_{\mathcal{C}})$, respectively. We often write $(E,\mathcal{C},r_{\mathcal{C}})$ simply as $\mathcal{C}$ when its ground set and rank function are clear. In addition, all the maximal sets in $\mathcal{C}$ have the same rank, which is denoted by $r(\mathcal{C})$ and called the {\em rank}  of the semimatroid $\mathcal{C}$. A set $X\in\mathcal{C}$ is {\em independent} if $r_{\mathcal{C}}(X)=|X|$, and {\em dependent} otherwise. Moreover, a maximal independent set and a minimal dependent set of $\mathcal{C}$ are  referred to as a {\em basis} and a {\em circuit} in turn. We use the notation $\mathscr{B}(\mathcal{C})$ to represent the set of bases of $\mathcal{C}$. In particular, an element $e\in\mathcal{C}$ is called an {\em isthmus} if $e$ is contained in every basis, and called a {\em loop} of $\mathcal{C}$ if $e$ is a circuit (i.e., $r_\mathcal{C}(e)=0$). The {\em closure} of $X$ in $\mathcal{C}$ is defined as
\[
{\rm cl}_{\mathcal{C}}(X):=\big\{x\in E\mid X\cup e\in\mathcal{C},\;r_{\mathcal{C}}(X\cup e)=r_{\mathcal{C}}(X)\big\}.
\]
Especially, a member $X\in\mathcal{C}$ is a {\em flat} of  the semimatroid $\mathcal{C}$ if ${\rm cl}_{\mathcal{C}}(X)=X$. Furthermore, the poset $\mathscr{L}(\mathcal{C})$, which consists of all flats of $\mathcal{C}$ ordered by set inclusion, forms a geometric semilattice, as presented in \cite[Theorem 6.4]{Ardila2007}. In fact, ${\rm cl}_{\mathcal{C}}(X)$ is the unique maximal element in $\mathcal{C}$ containing $X$ with rank $r_{\mathcal{C}}(X)$, as shown in the proof of {\rm (CLR1)} in \cite[Proposition 2.4]{Ardila2007}.

In this paper, we shall work with a fixed semimatroid $(E,\mathcal{C},r_{\mathcal{C}})$. Let $X\in\mathcal{C}$. The {\em restriction} $\mathcal{C}|X$ of $\mathcal{C}$ to $X$ or the {\em deletion} $\mathcal{C}\backslash (E-X)$ from $\mathcal{C}$, is a semimatroid $(X,\mathcal{C}|X,r_{\mathcal{C}|X})$ on the ground set $X$ with 
\begin{equation}\label{Deletion1}
\mathcal{C}|X=\mathcal{C}\backslash (E-X):=\{Y\subseteq X\mid Y\in\mathcal{C}\}
\end{equation}
and the rank function
\begin{equation}\label{Deletion-Rank-Semi}
r_{\mathcal{C}|X}(Y):=r_{\mathcal{C}}(Y),\quad\forall \;Y\in\mathcal{C}|X.
\end{equation}
The {\em contraction} $\mathcal{C}/X$ is a semimatroid $(E-X,\mathcal{C}/X,r_{\mathcal{C}/X})$ on the ground set $E-X$ with
\begin{equation}\label{Contraction1}
\mathcal{C}/X:=\{Y\subseteq E-X\mid Y\sqcup X\in\mathcal{C}\}
\end{equation}
and the rank function
\begin{equation}\label{Contraction-Rank-Semi}
r_{\mathcal{C}/X}(Y):=r_{\mathcal{C}}(Y\sqcup X)-r_{\mathcal{C}}(X),\quad\forall \;Y\in\mathcal{C}/X.
\end{equation}

The {\em Tutte polynomial} $T(\mathcal{C};\lambda,x)$  and  the {\em characteristic polynomial} $\chi(\mathcal{C};\lambda)$ of the semimatroid $\mathcal{C}$, introduced by Ardila \cite{Ardila2007}, are respectively defined by
\[
T(\mathcal{C};\lambda,x):=\sum_{A\in\mathcal{C}}(\lambda-1)^{r(\mathcal{C})-r_{\mathcal{C}}(A)}(x-1)^{|A|-r_{\mathcal{C}}(A)}
\]
and
\[
\chi(\mathcal{C};\lambda):=\sum_{A\in\mathcal{C}}(-1)^{|A|}\lambda^{r(\mathcal{C})-r_{\mathcal{C}}(A)}.
\]
In particular, both $T(\mathcal{C};\lambda,x)$ and $\chi(\mathcal{C};\lambda)$ are defined as $1$ when $E=\emptyset$. Obviously, $T(\mathcal{C};\lambda,x)$ is a generalization of  $\chi(\mathcal{C};\lambda)$ related by 
\[
\chi(\mathcal{C};\lambda)=(-1)^{r(\mathcal{C})}T(\mathcal{C};1-\lambda,0).
\]

\subsection{Deletion-contraction recurrences and basis activities expansions }\label{Sec-2-1}
Let $\bm x$ be a labeled multiset $\{x_e\mid e\in E\}$ of variables or numbers, where each element $e$ in the ground set $E$ of the semimatroid $\mathcal{C}$ is associated with a unique variable $x_e$. If $A\subseteq E$, then the notation $\bm x^A$ is given by the formula
\[
\bm x^A:=\prod_{e\in A}x_e.
\]
In addition, we have the related formulae
\[
(-\bm x)^A:=\prod_{e\in A}(-x_e)=(-1)^{|A|}\prod_{e\in A}x_e=(-1)^{|A|}\bm x^A
\]
and 
\[
(\bm x\bm y)^A:=\prod_{e\in A}x_ey_e=\bm x^A\bm y^A.
\]

Recall that the multivariate Tutte polynomial $Z(\mathcal{C};\lambda,\bm x)$ of the semimatroid $\mathcal{C}$ is given by
\[
Z(\mathcal{C};\lambda,\bm x)=\sum_{A\in\mathcal{C}}\lambda^{-r_\mathcal{C}(A)}\bm x^A.
\]
We now consider the deletion-contraction recurrence of the multivariate Tutte polynomial $Z(\mathcal{C};\lambda,\bm x)$. Notably, when restricted to matroids, this formula in \autoref{Mul-DCF} naturally coincides with the deletion-contraction formulae presented in \cite[(4.18a) and (4.18b)]{Sokal2005}. However, \autoref{Mul-DCF} holds for semimatroids, a broader class of objects that includes matroids but is not limited to them. 
\begin{theorem}[Deletion-contraction formula]\label{Mul-DCF}
Let $(E,\mathcal{C},r_{\mathcal{C}})$ be a semimatroid. For any $e\in E$, define $\bm x_{-e}:=\{x_f\}_{f\in E-e}$. Then 
\[
Z(\mathcal{C};\lambda,\bm x)=
\begin{cases}
(x_e+1)Z(\mathcal{C}\backslash e;\lambda,\bm x_{-e}),&\text{ if } e \text{ is a loop};\\
(\frac{x_e}{\lambda}+1)Z(\mathcal{C}\backslash e;\lambda,\bm x_{-e}),&\text{ if } e \text{ is an isthmus};\\
Z(\mathcal{C}\backslash e;\lambda,\bm x_{-e})+\frac{x_e}{\lambda}Z(\mathcal{C}/e;\lambda,\bm x_{-e}),&\text{ if } e \text{ is neither a loop nor an isthmus};\\
Z(\mathcal{C}\backslash e;\lambda,\bm x_{-e}),&\text{ if } e\in E \text{ and } e\notin\mathcal{C}.
\end{cases}
\]
\end{theorem}
\begin{proof}
For convenience, we can partition the set $\mathcal{C}$ into the following two parts associated with a fixed element $e\in E$: 
\[
\mathcal{C}_e:=\big\{A\in\mathcal{C}\mid e\in A\big\}\quad\And\quad\mathcal{C}_{-e}:=\big\{A\in\mathcal{C}\mid e\notin A\big\}.
\]
Immediately,  we can write the multivariate Tutte polynomial $Z(\mathcal{C};\lambda,\bm x)$ as
\[
Z(\mathcal{C};\lambda,\bm x)=\sum_{A\in\mathcal{C}_e}\lambda^{-r_\mathcal{C}(A)}\bm x^A+\sum_{A\in\mathcal{C}_{-e}}\lambda^{-r_\mathcal{C}(A)}\bm x^A.
\]

If $e$ is a loop, then we have the following relations:
\[
\mathcal{C}\backslash e=\mathcal{C}_{-e}=\big\{A-e\mid A\in \mathcal{C}_e\big\}
\]
from \eqref{Deletion1} and
\[
r_{\mathcal{C}}(A\sqcup e)=r_{\mathcal{C}}(A)=r_{\mathcal{C}\backslash e}(A),\quad\forall\, A\in \mathcal{C}\backslash e
\]
through \eqref{Deletion-Rank-Semi}. This implies
\begin{align*}
Z(\mathcal{C};\lambda,\bm x)&=\sum_{A\in\mathcal{C}\backslash e}\lambda^{-r_\mathcal{C}(A\sqcup e)}\bm x^{A\sqcup e}+\sum_{A\in\mathcal{C}\backslash e}\lambda^{-r_{\mathcal{C}\backslash e}(A)}\bm x^A\\
&=(x_e+1)Z(\mathcal{C}\backslash e;\lambda,\bm x_{-e}).
\end{align*}

If $e$ is an isthmus, similarly to the loop case, we obtain that
\[
\mathcal{C}\backslash e=\mathcal{C}_{-e}=\big\{A-e\mid A\in \mathcal{C}_e\big\}
\]
and
\[
r_{\mathcal{C}}(A)=r_{\mathcal{C}}(A-e)+1=r_{\mathcal{C}\backslash e}(A-e)+1,\quad\forall\, A\in \mathcal{C}_e.
\]
This means that
\begin{align*}
Z(\mathcal{C};\lambda,\bm x)&=\sum_{A\in\mathcal{C}\backslash e}\lambda^{-r_\mathcal{C}(A\sqcup e)}\bm x^{A\sqcup e}+\sum_{A\in\mathcal{C}\backslash e}\lambda^{-r_{\mathcal{C}\backslash e}(A)}\bm x^A\\
&=(\frac{x_e}{\lambda}+1)Z(\mathcal{C}\backslash e;\lambda,\bm x_{-e}).
\end{align*}

If $e$ is neither a loop nor an isthmus, we deduce that
\[
\mathcal{C}/e=\big\{A-e\mid A\in \mathcal{C}_e\big\}
\]
by \eqref{Contraction1} and
\[
r_{\mathcal{C}}(A\sqcup e)=r_{\mathcal{C}/e}(A)+1,\quad\forall\, A\in \mathcal{C}/e
\]
via \eqref{Contraction-Rank-Semi}. This further derive that 
\begin{align*}
Z(\mathcal{C};\lambda,\bm x)&=\sum_{A\in\mathcal{C}/ e}\lambda^{-r_\mathcal{C}(A\sqcup e)}\bm x^{A\sqcup e}+\sum_{A\in\mathcal{C}\backslash e}\lambda^{-r_{\mathcal{C}\backslash e}(A)}\bm x^A\\
&=\frac{x_e}{\lambda}Z(\mathcal{C}/e;\lambda,\bm x_{-e})+Z(\mathcal{C}\backslash e;\lambda,\bm x_{-e}).
\end{align*}

Finally, if $e\in E$ and $e\notin\mathcal{C}$, the case is trivial. This completes the proof.
\end{proof}

Next, we present a multivariate version of Crapo's formula and extend it to the multivariate Tutte polynomials of semimatroids with integral multiplicities. Let $B$ be a basis of the semimatroid $\mathcal{C}$. We review from \cite[Lemma 9.1]{Ardila2007} that if $e\in E-B$ such that $B\sqcup e\in\mathcal{C}$, then the set $B\sqcup e$ contains a unique circuit $C(e,B)$ of $\mathcal{C}$ and $e\in C(e,B)$. The circuit $C(e,B)$ is called the {\em fundamental circuit} of $e$ with respect to $B$. By duality, when $e\in B$, we call the unique cocircuit of $\mathcal{C}$ contained in the set $E-B\sqcup e$ the {\em fundamental cocircuit} of $e$ with respect to $B$, denoted by $C^*(e,B)$. Furthermore, let $\prec$ be a linear ordering of $E$. An element $e\in E-B$ is an {\em externally active element} for $B$ if $e$ is the minimal element of the fundamental circuit $C(e,B)$ with respect to $\prec$. Let ${\rm EA}(B)$ be the set of externally active elements  for $B$. Additionally, an element $e\in B$ is an {\em internally active element} in $B$ if $e$ is the minimal element of the fundamental cocircuit $C^*(e,B)$ with respect to $\prec$. Let ${\rm IA}(B)$ be the set of internally active elements in $B$.

To obtain our desired result, the following lemma is essential. It first appeared in \cite[Proposition 9.11]{Ardila2007}, and states that the simplicial complex $\mathcal{C}$ admits a disjoint union decomposition indexed by the bases of $\mathcal{C}$.
\begin{lemma}[\cite{Ardila2007}, Proposition 9.11]\label{Decomposition}
Let $(E,\mathcal{C},r_{\mathcal{C}})$ be a semimatroid and $\prec$ a linear ordering of $E$. Then
\[
\mathcal{C}=\bigsqcup_{B\in\mathscr{B}}\big[B-{\rm IA}(B),B\sqcup{\rm EA(B)}\big].
\]
\end{lemma}

By applying \autoref{Decomposition}, we give a simple proof of the basis activities expansion for the multivariate Tutte polynomial, which is a multivariate generalization of Crapo's formula. It should be noted that \autoref{Mul-EIA} agrees with the multivariate version for the multivariate Tutte polynomials of matroids in \cite[Corollary 4.7]{BM2014}. 
\begin{theorem}\label{Mul-EIA}
Let $(E,\mathcal{C},r_{\mathcal{C}})$ be a semimatroid. Then
\[
Z(\mathcal{C};\lambda,\bm x)=\lambda^{-r(\mathcal{C})}\sum_{B\in\mathscr{B}(\mathcal{C})}\prod_{b\in B}x_b\prod_{e\in {\rm EA}(B)}(x_e+1)\prod_{i\in {\rm IA}(B)}(\frac{\lambda}{x_i}+1).
\]
\end{theorem}
\begin{proof}
Note that from \cite[Lemma 9.12]{Ardila2007}, for any $I\subseteq{\rm IA}(B)$ and $S\subseteq{\rm EA}(B)$,
\[
r(\mathcal{C})-r_\mathcal{C}(B-I\sqcup S)=|I|.
\]
Combining  \autoref{Decomposition}, we deduce 
\begin{align*}
Z(\mathcal{C};\lambda,\bm x)&=\sum_{B\in\mathscr{B}(\mathcal{C})}\sum_{B-{\rm IA}(B)\subseteq A\subseteq B\sqcup{\rm EA}(B)}\lambda^{-r_\mathcal{C}(A)}\bm x^A\\
&=\lambda^{-r(\mathcal{C})}\sum_{B\in\mathscr{B}(\mathcal{C})}\bm x^B\sum_{B-{\rm IA}(B)\subseteq A\subseteq B-{\rm EA}(B)}\lambda^{r(\mathcal{C})-r_\mathcal{C}(A)}\bm x^{A-B}\\
&=\lambda^{-r(\mathcal{C})}\sum_{B\in\mathscr{B}(\mathcal{C})}\prod_{b\in B}x_b\sum_{I\subseteq{\rm IA}(B),\,S\subseteq{\rm EA}(B)}
\lambda^{r(\mathcal{C})-r_\mathcal{C}(B-I\sqcup S)}\bm x^{(B-I\sqcup S)-B}\\
&=\lambda^{-r(\mathcal{C})}\sum_{B\in\mathscr{B}(\mathcal{C})}\prod_{b\in B}x_b\sum_{S\subseteq{\rm EA}(B)}
\bm x^S\sum_{I\subseteq{\rm IA}(B)}\lambda^{|I|}\bm x^{-I}\\
&=\lambda^{-r(\mathcal{C})}\sum_{B\in\mathscr{B}(\mathcal{C})}\prod_{b\in B}x_b\prod_{e\in {\rm EA}(B)}(x_e+1)\prod_{i\in {\rm IA}(B)}(\frac{\lambda}{x_i}+1).
\end{align*}
This completes the proof.
\end{proof}

We continue to explore the connection between the multivariate Tutte polynomial and the standard Tutte polynomial of a semimatroid. Recall that when all the variables $x_e$ are set to the same variable $x$, the multivariate Tutte polynomial specializes to the following dichromatic polynomial:
\[
Z(\mathcal{C};\lambda,x)=\sum_{A\in\mathcal{C}}\lambda^{-r_\mathcal{C}(A)}x^{|A|}.
\]
This generalizes the dichromatic polynomial for graphs introduced by Tutte \cite{Tutte1967}, which is equivalent to the standard Tutte polynomial via a certain transformation presented in \eqref{Relation}. As an immediate consequence of \autoref{Mul-DCF} and \autoref{Mul-EIA}, we respectively obtain the deletion-contraction recurrence and a combinatorial expression for the dichromatic polynomial $Z(\mathcal{C};\lambda,x)$.
\begin{corollary}\label{Mul-Dichromatic-IEA}
Let $(E,\mathcal{C},r_{\mathcal{C}})$ be a semimatroid. Then the dichromatic polynomial $Z(\mathcal{C};\lambda,x)$ can be equivalently defined in the following two ways:
\[
Z(\mathcal{C};\lambda,x)=
\begin{cases}
(x+1)Z(\mathcal{C}\backslash e;\lambda,x),&\text{ if } e \text{ is a loop};\\
(\frac{x}{\lambda}+1)Z(\mathcal{C}\backslash e;\lambda,x),&\text{ if } e \text{ is an isthmus};\\
Z(\mathcal{C}\backslash e;\lambda,x)+\frac{x}{\lambda}Z(\mathcal{C}/e;\lambda,x),&\text{ if } e \text{ is neither a loop nor an isthmus};\\
Z(\mathcal{C}\backslash e;\lambda,x),&\text{ if } e\in E \text{ and } e\notin\mathcal{C}.
\end{cases}
\]
and
\[
Z(\mathcal{C};\lambda,x)=\lambda^{-r(\mathcal{C})}x^{r(\mathcal{C})}\sum_{B\in\mathscr{B}(\mathcal{C})}(\frac{\lambda}{x}+1)^{|{\rm IA(B)}|}(x+1)^{|{\rm EA}(B)|}.
\]
\end{corollary}

Furthermore, the dichromatic polynomial $Z(\mathcal{C};\lambda,x)$ and the standard Tutte polynomial $T(\mathcal{C};\lambda,x)$ are related by 
\begin{equation}\label{Relation}
T(\mathcal{C};\lambda,x)=(\lambda-1)^{r(\mathcal{C})}Z\big(\mathcal{C};(\lambda-1)(x-1);x-1\big).
\end{equation}
Applying \autoref{Mul-Dichromatic-IEA} to \eqref{Relation},  we directly recover the deletion-contraction formula and Crapo's formula for semimatroids, first established by Ardila \cite{Ardila2007}.
\begin{proposition}[\cite{Ardila2007}, Proposition 8.2 and Theorem 9.5]
Let $(E,\mathcal{C},r_{\mathcal{C}})$ be a semimatroid. Then the Tutte polynomial $T(\mathcal{C};\lambda,x)$ can be equivalently defined in the following two ways:
\[
T(\mathcal{C};\lambda,x)=
\begin{cases}
xT(\mathcal{C}\backslash e;\lambda,x),&\text{ if } e \text{ is a loop};\\
\lambda T(\mathcal{C}\backslash e;\lambda,x),&\text{ if } e \text{ is an isthmus};\\
T(\mathcal{C}/e;\lambda,x)+T(\mathcal{C}\backslash e;\lambda,x),&\text{ if } e \text{ is neither a loop nor an isthmus};\\
T(\mathcal{C}\backslash e;\lambda,x),&\text{ if } e\in E \text{ and } e\notin\mathcal{C}.
\end{cases}
\]
and
\[
T(\mathcal{C};\lambda,x)=\sum_{B\in\mathscr{B}(\mathcal{C})}\lambda^{|{\rm IA(B)}|}x^{|{\rm EA}(B)|}.
\]
\end{proposition}
\subsection{Convolution formulae}\label{Sec-2-2}
To better understand the relationship between the convolution identity \eqref{Convolution-Formula} for the Tutte polynomial and the convolution-multiplication identity \cite[Theorem 4]{Kung2004} for the multiplicative characteristic polynomial, Kung studied various convolution formulae for subset-corank polynomials and size-corank polynomials of matroids in \cite{Kung2010}. Building on Kung's work, this subsection is devoted to examining a number of convolution identities for multivariate Tutte polynomials and dichromatic polynomials of semimatroids. We start with two convolution identities for the multivariate Tutte polynomial.
\begin{theorem}\label{Mul-Convolution-Formula}
Let $(E,\mathcal{C},r_{\mathcal{C}})$ be a semimatroid. Then
\[
Z(\mathcal{C};\lambda\xi,\bm x\bm y)=\sum_{T\in\mathcal{C}}\xi^{-r_\mathcal{C}(T)}(-\bm y)^TZ(\mathcal{C}|T;\lambda,-\bm x)Z(\mathcal{C}/T;\xi,\bm y).
\]
In particular, we have
\[
Z(\mathcal{C};\lambda,\bm y)=\sum_{T\in\mathcal{C}}(-\bm y)^TZ(\mathcal{C}|T;\lambda,-1)Z(\mathcal{C}/T;1,\bm y).
\]
\end{theorem}
\begin{proof}
As $\mathcal{C}$ is a simplicial complex, for any $B\subseteq A$ in $\mathcal{C}$, the interval 
\[
[B,A]:=\{T\subseteq E\mid B\subseteq T\subseteq A\}
\]
is also contained in $\mathcal{C}$. Therefore, for any $B\subseteq A$ in $\mathcal{C}$, we have
\[
\sum_{B\subseteq T\subseteq A}(-1)^{|T-B|}=
\begin{cases}
1,&\text{ if } A=B;\\
0,&\text{otherwise}.
\end{cases}
\] 
Thus, we further deduce 
\begin{align*}
Z(\mathcal{C};\lambda\xi,\bm x\bm y)&=\sum_{A\in\mathcal{C}}(\lambda\xi)^{-r_\mathcal{C}(A)}(\bm x\bm y)^A\\
&=\sum_{B,A\in\mathcal{C}}\lambda^{-r_\mathcal{C}(B)}\bm x^B\xi^{-r_\mathcal{C}(A)}\bm y^A\Big[\sum_{B\subseteq T\subseteq A}(-1)^{|T-B|}\Big]\\
&=\sum_{T\in\mathcal{C}}\xi^{-r_\mathcal{C}(T)}(-\bm y)^T\Big[\sum_{B\subseteq T}\lambda^{-r_\mathcal{C}(B)}(-\bm x)^B\Big]\Big[\sum_{T\subseteq A\text{ in }\mathcal{C}}\xi^{r_\mathcal{C}(T)-r_\mathcal{C}(A)}\bm y^{A-T}\Big]\\
&=\sum_{T\in\mathcal{C}}\xi^{-r_\mathcal{C}(T)}(-\bm y)^TZ(\mathcal{C}|T;\lambda,-\bm x)Z(\mathcal{C}/T;\xi,\bm y).
\end{align*}

Moreover, we derive the following equation by setting $\xi=1$ and $x_e=1$ for all $e\in E$ in the first identity of \autoref{Mul-Convolution-Formula}:
\[
Z(\mathcal{C};\lambda,\bm y)=\sum_{T\in\mathcal{C}}(-\bm y)^TZ(\mathcal{C}|T;\lambda,-1)Z(\mathcal{C}/T;1,\bm y).
\]
This completes the proof.
\end{proof}

Furthermore, if $x_e=x$ and $y_e=y$ for each $e\in E$, we write $Z(\mathcal{C};\lambda\xi,xy)$ as shorthand for $Z(\mathcal{C};\lambda\xi,\bm x\bm y)$. We can then directly deduce the following specialization of \autoref{Mul-Convolution-Formula}.
\begin{corollary}\label{Mul-Convolution-Formula1}
Let $(E,\mathcal{C},r_{\mathcal{C}})$ be a semimatroid. Then
\[
Z(\mathcal{C};\lambda\xi,xy)=\sum_{T\in\mathcal{C}}\xi^{-r_\mathcal{C}(T)}(-y)^{|T|}Z(\mathcal{C}|T;\lambda,-x)Z(\mathcal{C}/T;\xi,y).
\]
In particular, we have
\[
Z(\mathcal{C};\lambda,y)=\sum_{T\in\mathcal{C}}(-y)^{|T|}Z(\mathcal{C}|T;\lambda,-1)Z(\mathcal{C}/T;1,y).
\]
\end{corollary}
It is worth remarking that the dichromatic polynomial $Z(\mathcal{C};\lambda,x)$ and the characteristic polynomial $\chi(\mathcal{C};\lambda)$ satisfy
\begin{equation}\label{Relation0}
Z(\mathcal{C};\lambda,-1)=\lambda^{-r(\mathcal{C})}\chi(\mathcal{C};\lambda).
\end{equation}
Together with \autoref{Mul-Convolution-Formula1}, we can obtain the next convolution-multiplication identity for the multiplicative characteristic polynomial of a semimatroid.
\begin{corollary}\label{Characteristic-Convolution-Formula}
Let $(E,\mathcal{C},r_{\mathcal{C}})$ be a semimatroid. Then
\[
\chi(\mathcal{C};\lambda\xi)=\sum_{T\in\mathcal{C}}\lambda^{r(\mathcal{C})-r_\mathcal{C}(T)}\chi(\mathcal{C}|T;\lambda)
\chi(\mathcal{C}/T;\xi).
\]
\end{corollary}
\begin{proof}
By setting $x=1$ and $y=-1$, and multiplying both sides of the convolution formula in \autoref{Mul-Convolution-Formula1} by $(\lambda\xi)^{r(\mathcal{C})}$, we obtain
\begin{align*}
(\lambda\xi)^{r(\mathcal{C})}Z(\mathcal{C};\lambda\xi,-1)&=(\lambda\xi)^{r(\mathcal{C})}\sum_{T\in\mathcal{C}}
\xi^{-r_\mathcal{C}(T)}Z(\mathcal{C}|T;\lambda,-1)Z(\mathcal{C}/T;\xi,-1)\\
&=\sum_{T\in\mathcal{C}}\lambda^{r(\mathcal{C})}Z(\mathcal{C}|T;\lambda,-1)\xi^{r(\mathcal{C})-r_\mathcal{C}(T)}Z(\mathcal{C}/T;\xi,-1).
\end{align*}
Note that 
\[
r(\mathcal{C})=r_\mathcal{C}(T)+r(\mathcal{C}/T)\quad\And\quad \chi(\mathcal{C};\lambda)=\lambda^{r(\mathcal{C})}Z(\mathcal{C};\lambda,-1).
\]
Thus, we deduce
\[
\chi(\mathcal{C};\lambda\xi)=\sum_{T\in\mathcal{C}}\lambda^{r(\mathcal{C})-r_\mathcal{C}(T)}\chi(\mathcal{C}|T;\lambda)
\chi(\mathcal{C}/T;\xi),
\]
which finishes the proof.
\end{proof}
Note that \autoref{Characteristic-Convolution-Formula} can be rewritten as
\[
\chi(\mathcal{C};\lambda\xi)=\sum_{T\in\mathscr{L}(\mathcal{C})}\lambda^{r(\mathcal{C})-r_\mathcal{C}(T)}\chi(\mathcal{C}|T;\lambda)
\chi(\mathcal{C}/T;\xi),
\]
which was first presented in \cite[Theorem 3.4]{Fu2025} and aligns with the convolution formula for the multiplicative
characteristic polynomial of matroids in \cite[Theorem 4]{Kung2004}. This is because if  $T\in\mathcal{C}$ is not a flat, then $\mathcal{C}/T$ contains a loop (see \cite[Lemma 3.6]{Fu2025}), forcing $\chi(\mathcal{C}/T;\xi)=0$.
\subsection{Weighted sums of polynomials}\label{Sec-2-3}
Extending identities 4 and 5 in \cite{Kung2010} to semimatroids, let us express the multivariate Tutte polynomial and the dichromatic polynomial as a weighted sum of other polynomials. 

First, recall an important property: for any semimatroid $(E,\mathcal{C},r_{\mathcal{C}})$,  the collection $\mathcal{C}$ of its central sets forms a simplicial complex. This leads to the following key lemma.
\begin{lemma}\label{Key-Property}
Let $(E,\mathcal{C},r_{\mathcal{C}})$ be a semimatroid and $A\in\mathcal{C}$. Then
\[
Z(\mathcal{C}|A;1,\bm x)=(\bm x+1)^A.
\]
\end{lemma}
\begin{proof}
Since $\mathcal{C}$ is a simplicial complex and $A\in\mathcal{C}$, every subset of $A$ belongs to $\mathcal{C}$. Following this, we have
\[
Z(\mathcal{C}|A;1,\bm x)=\sum_{T\subseteq A}\prod_{e\in T}x_e=\prod_{e\in A}(x_e+1)=(\bm x+1)^A,
\]
which completes the proof.
\end{proof}

\begin{theorem}\label{Weighted-Polynomial}
Let $(E,\mathcal{C},r_{\mathcal{C}})$ be a semimatroid. Then
\[
Z(\mathcal{C};\xi,\bm x\bm y)=\sum_{T\in\mathcal{C}}\xi^{-r_\mathcal{C}(T)}\bm y^T(\bm x+1)^TZ(\mathcal{C}/T;\xi,-\bm y).
\]
\end{theorem}
\begin{proof}
By setting $x_e=-x_e$, $y_e=-y_e$ for each $e\in E$, and $\lambda=1$ in the first identity of \autoref{Mul-Convolution-Formula}, we arrive at
\[
Z(\mathcal{C};\xi,\bm x\bm y)=\sum_{T\in\mathcal{C}}\xi^{-r_\mathcal{C}(T)}\bm y^TZ(\mathcal{C}|T;1,\bm x)Z(\mathcal{C}/T;\xi,-\bm y).
\]
Applying \autoref{Key-Property} to the above equation, we show that the result holds.
\end{proof}

Following \eqref{Relation0}, \autoref{Weighted-Polynomial} specializes to yield additional identities.
\begin{corollary}\label{Weighted-Polynomial1}
Let $(E,\mathcal{C},r_{\mathcal{C}})$ be a semimatroid. Then,
\begin{itemize}
\item[{\rm(1)}] $Z(\mathcal{C};\xi,\bm x)=\xi^{-r(\mathcal{C})}\sum_{T\in\mathscr{L}(\mathcal{C})}(\bm x+1)^T\chi(\mathcal{C}/T;\xi)$.
\item[{\rm(2)}] $Z(\mathcal{C};\xi,x)=\xi^{-r(\mathcal{C})}\sum_{T\in\mathscr{L}(\mathcal{C})}(x+1)^{|T|}\chi(\mathcal{C}/T;\xi)$.
\end{itemize}
\end{corollary}
\begin{proof}
When we set all the variables $y_e$ to the same value $1$  in \autoref{Weighted-Polynomial}, we have
\[
Z(\mathcal{C};\xi,\bm x)=\sum_{T\in\mathcal{C}}\xi^{-r_\mathcal{C}(T)}(\bm x+1)^TZ(\mathcal{C}/T;\xi,-1).
\]
Combining the following relations
\[ 
Z(\mathcal{C}/T;\xi,-1)=\xi^{-r(\mathcal{C}/T)}\chi(\mathcal{C}/T;\xi)\quad\And\quad r_\mathcal{C}(T)+r(\mathcal{C}/T)=r(\mathcal{C}),
\]
the above equation can be rewritten as
\[
Z(\mathcal{C};\xi,\bm x)=\xi^{-r(\mathcal{C})}\sum_{T\in\mathcal{C}}(\bm x+1)^T\chi(\mathcal{C}/T;\xi).
\]
Notice from \cite[Lemma 3.6]{Fu2025} that if $T\in\mathcal{C}$ is not a flat, then $\mathcal{C}/T$ contains a loop, forcing $\chi(\mathcal{C}/T;\xi)=0$. Thus, the range of the sums can be restricted to flats. This implies that part $(1)$ holds. Furthermore, by setting $x_e=x$ in part (1), we directly get part (2).
\end{proof}
\section{Applications}\label{Sec-3}
In 2010, Kung \cite{Kung2010} established a general convolution-multiplication identity for the multivariate and bivariate rank generating polynomials of graphs or matroids. Building on Kung's work, this section primarily introduces and investigates the subset-corank, size-corank, and rank generating polynomials of semimatroids. Note that the subset-corank polynomial is essentially equivalent to the multivariate Tutte polynomial under certain parametrizations, while both the size-corank and rank-generating polynomials are equivalent to the dichromatic polynomial. Therefore, we shall apply the results in \autoref{Sec-2} to these polynomial invariants, and then deduce the corresponding deletion-contraction formulae, basis activities expansions, and convolution formulae.  These findings generalize  the identities 1--5 in \cite{Kung2010} to semimatroids.

The {\em subset-corank polynomial} ${\bf SC}(\mathcal{C};\lambda,\bm x)$ of the semimatroid $\mathcal{C}$ is defined as
\[
{\bf SC}(\mathcal{C};\lambda,\bm x):=\sum_{A\in\mathcal{C}}\lambda^{r(\mathcal{C})-r_{\mathcal{C}}(A)}\Big[\prod_{e\in A}x_e\Big]=\sum_{A\in\mathcal{C}}\lambda^{r(\mathcal{C})-
r_{\mathcal{C}}(A)}\bm x^A.
\]
It is clear that the  subset-corank polynomial ${\bf SC}(\mathcal{C};\lambda,\bm x)$ and the multivariate Tutte polynomial $Z(\mathcal{C};\lambda,\bm x)$ are related by
\begin{equation}\label{Relation1}
{\bf SC}(\mathcal{C};\lambda,\bm x)=\lambda^{r(\mathcal{C})}Z(\mathcal{C};\lambda,\bm x).
\end{equation}

Applying \autoref{Mul-DCF}, \autoref{Mul-EIA}, \autoref{Mul-Convolution-Formula} and \autoref{Weighted-Polynomial} to \eqref{Relation1} immediately yields the corresponding properties of the subset-corank polynomial of a semimatroid. Since every matroid is a semimatroid, the properties established in \autoref{Pro-Subset-corank-Polynomial} hold for the subset-corank polynomials of matroids. Particularly, parts (3) and (4) in \autoref{Pro-Subset-corank-Polynomial} coincide with \cite[Identity 1, Identity 4  and Identity 5]{Kung2010} in this case.
\begin{corollary}\label{Pro-Subset-corank-Polynomial}
Let $(E,\mathcal{C},r_{\mathcal{C}})$ be a semimatroid. Then,
\begin{itemize}
\item[{\rm (1)}] Deletion-contraction formula:
\[
{\bf SC}(\mathcal{C};\lambda,\bm x)=
\begin{cases}
(x_e+1){\bf SC}(\mathcal{C}\backslash e;\lambda,\bm x_{-e}),&\text{ if } e \text{ is a loop};\\
(x_e+\lambda){\bf SC}(\mathcal{C}\backslash e;\lambda,\bm x_{-e}),&\text{ if } e \text{ is an isthmus};\\
{\bf SC}(\mathcal{C}\backslash e;\lambda,\bm x_{-e})+x_e{\bf SC}(\mathcal{C}/e;\lambda,\bm x_{-e}),&\text{ if } e \text{ is neither a loop nor an isthmus};\\
{\bf SC}(\mathcal{C}\backslash e;\lambda,\bm x_{-e}),&\text{ if } e\in E \text{ and } e\notin\mathcal{C}.
\end{cases}
\]
\item[{\rm (2)}] Basis activities expansion:
\[
{\bf SC}(\mathcal{C};\lambda,\bm x)=\sum_{B\in\mathscr{B}(\mathcal{C})}\prod_{b\in B}x_b\prod_{e\in {\rm EA}(B)}(x_e+1)\prod_{i\in {\rm IA}(B)}(\frac{\lambda}{x_i}+1).
\]
\item[{\rm (3)}] Convolution formulae:
\[
{\bf SC}(\mathcal{C};\lambda\xi,\bm {xy})=\sum_{T\in\mathcal{C}}\lambda^{r(\mathcal{C})-r_{\mathcal{C}}(T)}(-\bm y)^T{\bf SC}(\mathcal{C}|T;\lambda,-\bm x){\bf SC}(\mathcal{C}/T;\xi,\bm y),
\]
and
\[
{\bf SC}(\mathcal{C};\lambda,\bm y)=\sum_{T\in\mathcal{C}}\lambda^{r(\mathcal{C})-r_{\mathcal{C}}(T)}(-\bm y)^T{\bf SC}(\mathcal{C}|T;\lambda,-1){\bf SC}(\mathcal{C}/T;1,\bm y),
\]
\item[{\rm (4)}] Weighted sums of polynomials:
\[
{\bf SC}(\mathcal{C};\xi,\bm x\bm y)=\sum_{T\in\mathcal{C}}\bm y^T(\bm x+1)^T{\bf SC}(\mathcal{C}/T;\xi,-\bm y),
\]
and
\[
{\bf SC}(\mathcal{C};\xi,\bm x)=\sum_{T\in\mathscr{L}(\mathcal{C})}(\bm x+1)^T\chi(\mathcal{C}/T;\xi).
\]
\end{itemize}
\end{corollary}

When all the variables $x_e$ are set to the same variable $x$, the subset-corank polynomial specializes to the size-corank polynomial. 
The {\em size-corank polynomial }${\rm SC}(\mathcal{C};\lambda,x)$ of the semimatroid $\mathcal{C}$ is the polynomial in two variables $\lambda$ and $x$, defined by
\[
{\rm SC}(\mathcal{C};\lambda,x):=\sum_{A\in\mathcal{C}}\lambda^{r(\mathcal{C})-r_{\mathcal{C}}(A)}x^{|A|}.
\]

Setting $x_e=x$ and $y_e=y$ for each $e\in E$ in \autoref{Pro-Subset-corank-Polynomial}, we directly obtain the next corollary. Likewise, the size-corank polynomial of a matroid exhibits the same properties listed in \autoref{Pro-Size-corank-Polynomial}. Notably, parts (3) and (4) agree with \cite[Identity 2 and Identity 5]{Kung2010}.
\begin{corollary}\label{Pro-Size-corank-Polynomial}
Let $(E,\mathcal{C},r_{\mathcal{C}})$ be a semimatroid. Then,
\begin{itemize}
\item[{\rm (1)}] Deletion-contraction formula:
\[
{\rm SC}(\mathcal{C};\lambda,x)=
\begin{cases}
(x+1){\rm SC}(\mathcal{C}\backslash e;\lambda, x),&\text{ if } e \text{ is a loop};\\
(x+\lambda){\rm SC}(\mathcal{C}\backslash e;\lambda,x),&\text{ if } e \text{ is an isthmus};\\
{\rm SC}(\mathcal{C}\backslash e;\lambda,x)+x{\rm SC}(\mathcal{C}/e;\lambda,x),&\text{ if } e \text{ is neither a loop nor an isthmus};\\
{\rm SC}(\mathcal{C}\backslash e;\lambda,x),&\text{ if } e\in E \text{ and } e\notin\mathcal{C}.
\end{cases}
\]
\item[{\rm (2)}]Basis activities expansion:
\[
{\rm SC}(\mathcal{C};\lambda,x)=x^{r(\mathcal{C})}\sum_{B\in\mathscr{B}(\mathcal{C})}(\frac{\lambda}{x}+1)^{|{\rm IA(B)}|}(x+1)^{|{\rm EA}(B)|}.
\]
\item[{\rm (3)}] Convolution formulae:
\[
{\rm SC}(\mathcal{C};\lambda\xi,xy)=\sum_{T\in\mathcal{C}}\lambda^{r(\mathcal{C})-r_{\mathcal{C}}(T)}(-y)^{|T|}{\rm SC}(\mathcal{C}|T;\lambda,-x){\rm SC}(\mathcal{C}/T;\xi,y),
\]
and
\[
{\rm SC}(\mathcal{C};\lambda,y)=\sum_{T\in\mathcal{C}}\lambda^{r(\mathcal{C})-r_{\mathcal{C}}(T)}(-y)^{|T|}{\rm SC}(\mathcal{C}|T;\lambda,-1){\rm SC}(\mathcal{C}/T;1,y),
\]
\item[{\rm (4)}] Weighted sums of polynomials:
\[
{\rm SC}(\mathcal{C};\xi,x)=\sum_{T\in\mathscr{L}(\mathcal{C})}(x+1)^{|T|}\chi(\mathcal{C}/T;\xi).
\]
\end{itemize}
\end{corollary}

Naturally, we can extend the rank generating polynomial for graphs and matroids to semimatroids. The {\em rank generating polynomial} $R(\mathcal{C};\lambda,x)$ of the semimatroid $\mathcal{C}$ is defined by 
\[
R(\mathcal{C};\lambda,x):=\sum_{A\in\mathcal{C}}\lambda^{r(\mathcal{C})-r_{\mathcal{C}}(A)}x^{|A|-r_{\mathcal{C}}(A)}.
\]
It is obvious that the size-corank polynomial and the rank generating polynomial are closely related by
\[
x^{-r(\mathcal{C})}{\rm SC}(\mathcal{C};x\lambda,x)=R(\mathcal{C};\lambda,x)
\]
and
\begin{equation}\label{Relation3}
{\rm SC}(\mathcal{C};\lambda,x)=x^{r(\mathcal{C})}R(\mathcal{C};\lambda/x,x).
\end{equation}
Both polynomials specialize to the characteristic polynomial. Specifically,
\[
{\rm SC}(\mathcal{C};\lambda,-1)=\chi(\mathcal{C};\lambda)\quad\And\quad R(\mathcal{C};-\lambda,-1)=(-1)^{r(\mathcal{C})}\chi(\mathcal{C};\lambda).
\]

Applying \autoref{Pro-Size-corank-Polynomial} to \eqref{Relation3}, we directly derive that the rank generating polynomial has properties analogous to those of the size-corank polynomial.
\begin{corollary}\label{Prop-Rank-Polynomial}
Let $(E,\mathcal{C},r_{\mathcal{C}})$ be a semimatroid. Then,
\begin{itemize}
\item[\rm{(1)}] Deletion-contraction formula:
\[
R(\mathcal{C};\lambda,x)=
\begin{cases}
(x+1)R(\mathcal{C}\backslash e;\lambda,x),&\text{ if } e \text{ is a loop};\\
(\lambda+1)R(\mathcal{C}\backslash e;\lambda,x),&\text{ if } e \text{ is an isthmus};\\
R(\mathcal{C}/e;\lambda,x)+R(\mathcal{C}\backslash e;\lambda,x),&\text{ if } e \text{ is neither a loop nor an isthmus};\\
R(\mathcal{C}\backslash e;\lambda,x),&\text{ if } e\in E \text{ and } e\notin\mathcal{C}.
\end{cases}
\]
\item[\rm{(2)}] Basis activities expansion:
\[
R(\mathcal{C};\lambda,x)=\sum_{B\in\mathscr{B}(\mathcal{C})}(\lambda+1)^{|{\rm IA(B)}|}(x+1)^{|{\rm EA}(B)|}.
\]
\item[\rm{(3)}] Convolution formulae: 
\[
R(\mathcal{C};\lambda\xi,xy)=\sum_{T\in\mathcal{C}}\lambda^{r(\mathcal{C})-r_{\mathcal{C}}(T)}(-y)^{|T|-r_\mathcal{C}(T)}
R(\mathcal{C}|T;-\lambda,-x)R(\mathcal{C}/T;\xi,y),
\]
\[
R(\mathcal{C};\lambda,y)=\sum_{T\in\mathcal{C}}(-\lambda)^{r(\mathcal{C})-r_{\mathcal{C}}(T)}(-y)^{|T|-r_\mathcal{C}(T)}
R(\mathcal{C}|T;\lambda,-1)R(\mathcal{C}/T;-1,y),
\]
and
\begin{equation}\label{Semi-RGPCF}
R(\mathcal{C};\xi,x)=\sum_{T\in\mathcal{C}}R(\mathcal{C}|T;-1,x)R(\mathcal{C}/T;\xi,-1).
\end{equation}
\end{itemize}
\end{corollary}

Moreover, it follows from \cite[Lemma 3.6]{Fu2025} that for any $T\in\mathcal{C}$, $\mathcal{C}/T$ has no loops if and only if $T$ is a flat of $\mathcal{C}$. Consequently, the identity in \eqref{Semi-RGPCF} can be further simplified to the form given in the subsequent corollary.

\begin{corollary}\label{Semi-RGPCF1}
Let $(E,\mathcal{C},r_{\mathcal{C}})$ be a semimatroid. Then,
\begin{align*}
R(\mathcal{C};\xi,x)&=\sum_{T\in\mathscr{L}(\mathcal{C})}R(\mathcal{C}|T;-1,x)R(\mathcal{C}/T;\xi,-1)\\
&=\sum_{T\text{ is a cyclic flat of }\mathcal{C}}R(\mathcal{C}|T;-1,x)R(\mathcal{C}/T;\xi,-1),\\
\end{align*}
where a flat $T$ of $\mathcal{C}$ is a cyclic flat if it has no isthmuses.
\end{corollary}
\begin{proof}
From the deletion-contraction formula in part (1) of \autoref{Prop-Rank-Polynomial}, we directly derive that $R(\mathcal{C};\xi,-1)=0$ if  $\mathcal{C}$ has a loop; and $R(\mathcal{C};-1,x)=0$ if  $\mathcal{C}$ contains an isthmus. Applying  these results to the identity in \eqref{Semi-RGPCF}, we confirm that both identities in \autoref{Semi-RGPCF1} hold.
\end{proof}

Additionally, the Tutte polynomial $T(\mathcal{C};\lambda,x)$ and the rank generating polynomial $R(\mathcal{C};\lambda,x)$ satisfy the identity
\[
T(\mathcal{C};\lambda,x)=R(\mathcal{C};\lambda-1,x-1). 
\]
which describes their fundamental relationship. Based on this, using \autoref{Prop-Rank-Polynomial} and \autoref{Semi-RGPCF1}, we immediately obtain the deletion-contraction formula and basis activities expansion for the Tutte polynomial from \cite[Proposition 8.2, Theorem 9.5]{Ardila2007} and the convolution formulae for the Tutte polynomial from \cite[Theorem 3.5, Corollary 3.7]{Fu2025}.

Again, we remark that matroids are specialized subclasses of semimatroids. Naturally, we conclude this paper by noting that when these polynomial invariants are restricted to matroids, the previously established properties still hold.
\section*{Acknowledgements}
The work is supported by National Natural Science Foundation of China (12301424).

\end{document}